\documentclass[12pt,a4paper]{article}
\usepackage{indentfirst}
\usepackage{amsmath,amssymb,amsfonts,amsthm,graphics}
\usepackage{makeidx}
\usepackage{color}

\newtheorem{theorem}{Theorem}

\newtheorem{lemma}{Lemma}

\newtheorem{corollary}{Corollary}

\newtheorem{remark}{Remark}
\newtheorem{observation}{Observation}

\begin{document}
\title{\Large\bf Upper bound for the rainbow connection number of bridgeless
graphs with diameter 3\footnote{Supported by NSFC
No.11071130.}}
\author{\small Hengzhe Li, Xueliang Li, Yuefang Sun
\\
\small Center for Combinatorics and LPMC-TJKLC
\\
\small Nankai University, Tianjin 300071, China
\\
\small lhz2010@mail.nankai.edu.cn; lxl@nankai.edu.cn\\
\small syf@cfc.nankai.edu.cn}
\date{}
\maketitle
\begin{abstract}
A path in an edge-colored graph $G$, where adjacent edges may have
the same color, is called rainbow if no two edges of the path are
colored the same. The rainbow connection number $rc(G)$ of $G$ is
the smallest integer $k$ for which there exists a $k$-edge-coloring
of $G$ such that every pair of distinct vertices of $G$ is connected
by a rainbow path. It is known that for every integer $k\geq 2$
deciding if a graph $G$ has $rc(G)\leq k$ is NP-Hard, and a graph
$G$ with $rc(G)\leq k$ has diameter $diam(G)\leq k$. In foregoing
papers, we showed that a bridgeless graph with diameter $2$ has
rainbow connection number at most $5$. In this paper, we prove that
a bridgeless graph with diameter $3$ has rainbow connection number
at most $9$. We also prove that for any bridgeless graph $G$ with
radius $r$, if every edge of $G$ is contained in a triangle, then
$rc(G)\leq 3r$. As an application, we get that for any graph $G$
with minimum degree at least $3$, $rc(L(G))\leq 3 rad(L(G))\leq 3
(rad(G)+1)$.

\end{abstract}

\section{Introduction}

All graphs in this paper are undirected, finite and simple. We refer
to book \cite{bondy} for notation and terminology not described
here. A path $u=u_1,u_2,\ldots,u_k=v$ is called a $P_{u,v}$ path.
The length of a path is its number of edges. The distance between
two vertices $x$ and $y$ in $G$, denoted by $d(x,y)$, is the length
of a shortest path between them. The $eccentricity$ of a vertex $x$
is $ecc(x)=max_{y\in V(G)}d(x,y)$. The $radius$ and $diameter$ of
$G$ are $rad(G)=min_{x\in V(G)}ecc(x)$ and $diam(G)=max_{x\in
V(G)}ecc(x)$, respectively. A vertex $u$ is a $center$ if
$ecc(u)=rad(G)$.

A path in an edge-colored graph $G$, where adjacent edges may have
the same color, is called $rainbow$ if no two edges of the path are
colored the same. An edge-coloring of graph $G$ is a $rainbow\
edge$-$coloring$ if every pair of distinct vertices of graph $G$ is
connected by a rainbow path. The $rainbow\ connection\ number\
rc(G)$ of $G$ is the minimum integer $k$ for which there exists a
$k$-edge-coloring of $G$ such that every pair of distinct vertices
of $G$ is connected by a rainbow path. It is easy to see that
$diam(G)\leq rc(G)$ for any connected graph $G$.

The rainbow connection number was introduced by Chartrand et al. in
\cite{char}. It is of great use in transferring information of high
security in multicomputer networks. We refer the readers to
\cite{chak,char2} for details.

Chakraborty et al. \cite{chak} investigated the hardness and
algorithms for the rainbow connection number, and showed that given
a graph $G$, deciding if $rc(G)=2$ is NP-Complete. Recently, Ananth
and Nasre \cite{AN} showed that for every integer $k\geq 2$ deciding
if a graph has $rc(G)\leq k$ is NP-Hard. In particular, computing
$rc(G)$ is NP-Hard. Bounds for the rainbow connection number of a
graph have also been studies in terms of other graph parameters, for
example, radius, diameter, dominating number, minimum degree,
connectivity, etc.
\cite{bas,chan,char,char2,kri,lili,lili2,sch,sch2}. A survey on the
rainbow connection number is given by Li and Sun in \cite{lis}.

A graph $G$ is called chordal if $G$ has no induced cycle of length
greater $3$. Chandran, Das, Rajendraprasad and Varma in \cite{chan}
got the following bound for the rainbow connection number of a graph
$G$ in terms of its radius.
\begin{theorem}{\upshape\cite{chan}}
Let $G$ be a bridgeless chordal graph. Then $rc(G)\leq 3rad(G)$.
\end{theorem}

Let $G$ be a bridgeless chordal graph and $e$ be any edge of $G$.
Since $G$ is bridgeless, $e$ must be contained in some cycles of
$G$. Let $C$ be a smallest cycle containing $e$. Then $C$ has to be
a triangle; otherwise, $G$ admits an induced cycle with length at
least $4$, which contradicts to the fact that $G$ is a chordal
graph. So every edge of a bridgeless chordal graph is contained in a
triangle. Thus, our following theorem is a generalization of
Theorem~1.
\begin{theorem}
Let $G$ be a bridgeless graph. If every edge of $G$ is contained in
a triangle, then $rc(G)\leq 3rad(G)$.
\end{theorem}

Consequently, we have the following corollary.
\begin{corollary}
Let $G$ be a graph with minimum degree at least $3$. Then
$rc(L(G))\leq 3 rad(L(G))\leq 3  (rad(G)+1)$.
\end{corollary}

Since for every $k=2, 3$ deciding if a graph $G$ has $rc(G)=k$ is
NP-Hard, and a graph $G$ with $rc(G)=2, 3$ has $diam(G)\leq 2, 3$,
respectively, to study graphs with a small diameter 2 or 3 is of
significance. In a foregoing papers \cite{lili, DL}, we got the
following result.
\begin{theorem}{\upshape\cite{lili, DL}}
Let $G$ be a bridgeless graph with diameter $2$. Then $rc(G)\leq 5$,
and the upper bound is sharp.
\end{theorem}

In this paper, we will show the following result.
\begin{theorem}
Let $G$ be a bridgeless graph $diam(G)=3$. Then $rc(G)\leq 9$.
\end{theorem}

Since $rad(G)\leq diam(G)$, from Theorem~2 we get that if $G$ is a
bridgeless graph with $diam(G)=3$ such that every edge of $G$ is
contained in a triangle, then $rc(G)\leq 3diam(G)=9$. So, to show
Theorem~3, we only need to show the following result.
\begin{theorem}
Let $G$ be bridgeless graph with $diam(G)=3$. If $G$ possesses an
edge not contained in any triangle, then $rc(G)\leq 9$.
\end{theorem}

This paper is organized as follows. In Section~$2$, we prove
Theorem~2. In Subsection~$3.1$, we present a vertex set partition of
a bridgeless graph with $diam(G)=3$, which admits an edge not
contained in any triangle, together with a $9$-edge-coloring under
this partition. In Subsection~$3.2$, we show that the above
$9$-edge-coloring is rainbow and get Theorem~5.

\section{Proof of Theorem~$2$}

Let $S$ be a set of vertices. We use $G[S]$ to denote the graph
$(S,E')$, which is called the induced subgraph of $G$ by $S$, where
$E'$ is the set of edges of $G$ whose two ends are contained in $S$.
Let $S$ and $S'$ be two disjoint vertex sets. We use $E[S,S']$ to
denote the set of edges having one endpoint in each one of $S$ and
$S'$.

\noindent{\itshape Proof of Theorem~$2$:} Let $x$ be a center in
$G$, $N_i(x)=\{y\ |\ d(x,y)=i\}$, $N_i[x]=\{y\ |\ d(x,y)\leq i\}$,
and $G_i=G[N_i[x]],\ i=1,2,\ldots,r$. Specially, $N(x)$ simply for
$N_1(x)$. We shall color the edges of $G$ by the following $r$ steps
such that $G$ is rainbow connected.

{\itshape First step:} We claim that $G[N(x)]$ has no isolated
vertex. Given $y\in N(x)$, let $z$ be a common neighbor of $x$ and
$y$. Clearly, $z\in N(x)$. So $G[N(x)]$ has no isolated vertex. Thus
$G[N(x)]$ has a spanning forest $F_1$ without isolated vertex. Let
$X_1$ and $Y_1$ be any one of the bipartition defined by this forest
$F_1$. We provide a $3$-edge-coloring $c_1: E(G[N[x]])\rightarrow
\{1,2,3\}$ of $G[N[x]]$ as follows: $c_1(e)=1$ if $e\in E[x,X_1]$;
$c_1(e)=2$ if $e\in E[x,Y_1]$; $c_1(e)=3$ if $e\in E[X_1,Y_1]$. We
show that $c_1$ is a rainbow edge-coloring of $G[N[x]]$. Pick any
two distinct vertices $u$ and $v$ in $G[N[x]]$. If one of $u$ and
$v$ is $x$, then $u,v$ is a rainbow path. If $u\in X$ and $v\in Y$,
then $u,x,v$ is a rainbow path connecting $u$ and $w$. Thus we
suppose $u,v\in X$ or $u,v\in Y$, without loss of generality, say
$u,v\in X$. Then $u,w,x,v$ is a rainbow path connecting $u$ and $v$,
where $w$ is a neighbor of $u$ in $F_1$ (contained in $Y$).
Therefore, $c_1$ is a rainbow coloring of $G_1$.

{\itshape Second step:} We claim that every isolated vertex of
$G[N_2(x)]$ admits at least two neighbors in $N(x)$. Given an
isolated vertex $y\in N_2(x)$, let $z$ be a neighbor of $y$ in
$N(x)$ and $w$ be a common neighbor of $y$ and $z$. If $w\in
N_2(x)$, then $z$ admits a neighbor in $G[N_2(x)]$, a contradiction.
Otherwise, $w\in N(x)$, that is, $z$ admits at least two neighbors
in $N(x)$. Thus, every isolated vertex of $G[N_2(x)]$ admits at
least two neighbors in $N(x)$. Let $Z_2=\{z\ |\ z\ is\ an\ isolated\
vertex\ in\ G[N_2(x)]\}$, $F_2$ be a spanning forest in
$G[N_2(x)\setminus Z_2]$, and $X_2,Y_2$ be any one of the
bipartition defined by this forest $F_2$. We provide a
$3$-edge-coloring $c_2: E(G[N_2(x)])\cup E[N(x),N_2(x)]\rightarrow
\{4,5,6\}$ defined as follows: $c_2(e)=4$ if $e\in E[N(x),X_2]$;
$c_2(e)=5$ if $e\in E[N(x),Y_2]$; $c_2(e)=6$ if $e\in E[X_2,Y_2]$;
for any vertex $z\in Z$, color one incident edge in $E[z,N(x)]$ by
color $4$, the others in $E[z,N(x)]$ by color $5$. It is easy to
check that $c_1\cup c_2$ is a rainbow coloring of $G_2$.

By similar arguments, we can give $G_i,\, i=1,2,\ldots,r$, a
$3i$-rainbow coloring. So $rc(G)=rc(G_r)\leq
3r$.\hfill$\sqcap\hskip-0.7em\sqcup$

\noindent{\bf Example~1:} Let $r$ and $i$ be positive integers such
that $1\leq i\leq (3r-1)^r+1$, $P^i_r$ be a path with vertex set
$\{u^i_0,u^i_1,\ldots,u^i_r\}$, where $u^i_{j_1}$ and $u^i_{j_2}$
are adjacent if and only if $|j_1-j_2|=1$. $H^i_r$ is a graph
obtained from $P^i_r$ by adding vertices $v^i_j$ and edges
$v^i_{j}u^i_{j-1}$ and $v^i_ju^i_j,\ j=1,2,\ldots,r$. Furthermore,
let $G$ be the graph obtained from $H^i_r$, where $1\leq i\leq
(3r-1)^r+1$, by identifying $u^i_0$ as a vertex $u_0$. Clearly, $G$
has radius $r$, and $u_0$ is a center of $G$.

We claim that $G$ has no $3r-1$-rainbow coloring. Let $c$ be any
$r-1$-edge-coloring of $G$. For any path with length $r$, there
exist at most $(3r-1)^r$ distinct colorings. Since there exist
$(3r-1)^r+1$'s paths $P^i_r$, at least two, without loss of
generality, say $P^1_r$ and $P^2_r$, have the some coloring, that
is, $u^1_ju^1_{j+1}$ and $u^2_ju^2_{j+1}$ have the same color. Thus
any rainbow path between $u^1_r$ and $u^2_r$ contains at most one of
$u^1_ju^1_{j+1}$ and $u^2_ju^2_{j+1}$. Clearly, any path satisfying
the above requirement has a length at least $3r$. Therefore,
$rc(G)\geq 3r-1$.

Now we show a $3r$-rainbow coloring $c$ of $G$ as follows:
$c(u^i_j,u^i_{j+1})=3j+1$, $c(u^i_j,v^i_{j+1})=3j+2$, and
$c(u^i_{j+1},u^i_{j+1})=3(j+1),\ j=0,2,\ldots,r-1$,
$i=1,2,\ldots,(3r-1)^k$. It is easy to check that the above
$3r$-edge-coloring is a rainbow edge-coloring of $G$.

\section{Proof of Theorem~5}

Let $c$ be a rainbow edge-coloring of $G$. If an edge $e$ is colored
by $i$, we say that $e$ is an $i$-$color\ edge$. Let $P$ be a
rainbow path. If $c(e)\in \{i_1,i_2,\ldots,i_r\}$ for any $e\in
E(P)$, then $P$ is called an $\{i_1,i_2,\ldots,i_r\}$-$rainbow\
path$.

\subsection{Vertex set partition and edge-coloring}

In this subsection, Let $e=uv$ be an edge not contained in any
triangle, and $A$ and $B$ denote $N(u)\setminus \{v\}$ and
$N(v)\setminus \{u\}$, respectively. Furthermore, $X,Y$ and $Z$
denote the sets $\{x\in N(A)\setminus N(B)\, |\, x\not\in A\cup
B\cup \{u,v\}\}$, $\{x\in N(B)\setminus N(A)\, |\, x\not\in A\cup
B\cup \{u,v\}\}$, and $\{x\in N(A)\cap N(B)\}$, respectively.

Let $D=A\cup B\cup X\cup Y\cup Z\cup \{u,v\}$. For $x\in
V(G)\setminus D$, we have $d(x,u)=d(x,v)=3$ since $diam(G)=3$, that
is, $N(x)\cap N(X)\neq\emptyset$ and $N(x)\cap N(Y)\neq\emptyset$.
We partition this set based on the distribution of the neighbors of
$x$.
\begin{align*}
& W=(N(X)\cap N(Y))\setminus D;\\
& I=(N(X)\cap N(Z))\setminus (W\cup D);\\
& K=(N(Y)\cap N(Z))\setminus (W\cup I\cup D);\\
& J=V(G)\setminus (W\cup I\cup K\cup D).
\end{align*}
See Figure $1$ for details. At this point, we partition $A,B,X$ and
$Y$ as follows:
\begin{align*}
& A_1=\{x\in A\ |\ x\ \mathrm{has\ neighbors\ in}\ B\cup X\cup Z\},\\
& A_2=\{x\in A\setminus A_1\ |\ x\ \mathrm{is\ not\ an \ isolated\ vertex\ in}\ G[A\setminus A_1]\},\\
& A_3=A\setminus (A_1\cup A_2),\\
& B_1=\{x\in B\ |\ x\ \mathrm{has\ neighbors\ in}\ A\cup Y\cup Z\},\\
& B_2=\{x\in B\setminus A_1\ |\ x\ \mathrm{is\ not\ an \ isolated\ vertex\ in}\ G[B\setminus B_1]\},\\
& B_3=B\setminus (B_1\cup B_2),\\
& X_1=\{x\in X\ |\ x\ \mathrm{has\ neighbors\ in}\ Y\cup Z\cup I\cup W\}.\\
& X_2=\{x\in X\setminus X_1\ |\ x\ \mathrm{is\ not\ an \ isolated\ vertex\ in}\ G[X\setminus X_1]\},\\
& X_3=\{x\in X\setminus (X_1\cup X_2)\ |\ N(x)\in A\},\\
& X_4=X\setminus (X_1\cup X_2\cup X_3),\\
& Y_1=\{y\in Y\ |\ y\ \mathrm{has\ neighbors\ in}\ X\cup Z\cup K\cup W\},\\
& Y_2=\{x\in Y\setminus Y_1\ |\ x\ \mathrm{is\ not\ an \ isolated\ vertex\ in}\ G[Y\setminus Y_1]\},\\
& X_3=\{x\in Y\setminus (Y_1\cup Y_2)\ |\ N(x)\in B\},\\
& X_4=Y\setminus (Y_1\cup Y_2\cup Y_3).
\end{align*}

\begin{figure}[h,t,b]
\begin{center}
\scalebox{0.8}[0.8]{\includegraphics{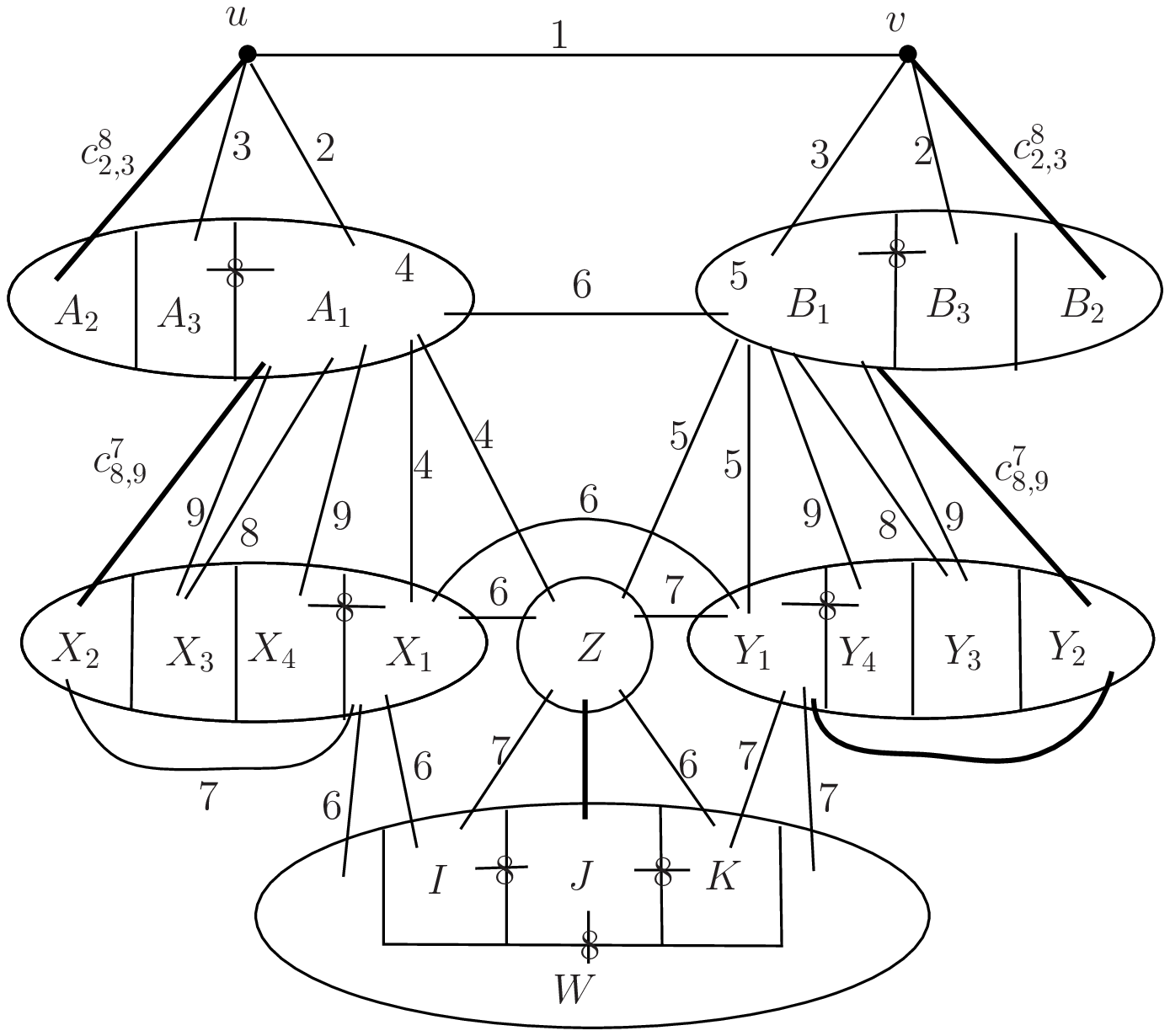}} \vspace*{20pt}

Figure 1. An part edge-coloring of $G$
\end{center}
\end{figure}

Note that, in Figure~$1$, if $x$ and $y$ lie in distinct ellipses
and there exists no edge joining the two ellipses, then $x$ and $y$
are nonadjacent in $G$. In general, the edges drawn in Figure~$1$ do
not represent that the subgraph is complete bipartite. The following
observation holds for the above vertex set partition since $G$ is
bridgeless.

\begin{observation}
$(1)$ For every $x\in A_3$, $N(x)\cap A_1\neq\emptyset$. $(2)$ For
every $x\in B_3$, $N(x)\cap B_1\neq\emptyset$. $(3)$ For every $x\in
X_4$, $N(x)\cap X_1\neq\emptyset$. $(4)$ For every $x\in Y_4$,
$N(x)\cap Y_1\neq\emptyset$.
\end{observation}

We give a partial $8$-edge-coloring of $G$ as follows:

$$c(e)=\left\{\begin{array}{ll}
1, & if\ e=uv;\\
2, & if\ e\in E[u,A_1]\cup E[v,B_3];\\
3, & if\ e\in E[u,A_3]\cup E[v,B_1];\\
4, & if\ e\in E[A_1,X_1\cup Z]\cup E[B_1,Y_4]\cup E(G[A_1]);\\
5, & if\ e\in E[A_1,X_4]\cup E[B_1,Y_1\cup Z]\cup E(G[B_1]);\\
6, & if\ e\in E[A_1,B_1]\cup E[Z,K]\cup E[X_1,I\cup Z\cup W\cup Y_1];\\
7, & if\ e\in E[Z,I]\cup E[Y_1,K\cup Z\cup W];\\
8, & if\ e\in E[A_1,A_3]\cup E[B_1,B_3]\cup E[X_1,X_4]\cup
E[Y_1,Y_4]\cup E[J,I\cup K\cup W];\\
9, & if\ e\in E[A_1,X_4]\cup E[B_1,Y_4].
 \end{array}\right.$$

\noindent See Figure~1 for details. For any $x\in X_3$, if $N(x)\cap
N(B_1)\neq \emptyset$, color $xx'$ by color $7$, where $x'\in
N(x)\cap N(B_1)$, color the others by color $9$. Otherwise $N(x)\cap
N(B_1)= \emptyset$, since $diam(G)=3$, there exists a neighbor $x'$
of $x$ such that $N(x')\cap N(B_1)\neq \emptyset$. Note that $x'\in
A_1$. Then we color $xx'$ by color $7$, color the others by color
$9$. Similarly, we color the edges in $E[B_1,Y_3]$ by colors $7$ and
$9$.

\begin{observation}
$(1)$ For any $x\in X_3$, there is a $\{3,5,6,7,8\}$ or
$\{3,5,6,7,9\}$-rainbow path between $x$ and $v$ under the above
edge-coloring. $(2)$ For any $y\in Y_3$, there is a $\{2,5,6,7,8\}$
or $\{2,5,6,7,9\}$-rainbow path between $y$ and $u$ under the above
edge-coloring.
\end{observation}

\begin{proof}
Since the proofs are similar, we only show $(1)$. Let $xx'$ be the
edge that is colored by color $7$ in the above coloring process.
Note that $x'\in A_1$.

If $x'$ has a neighbor, say $y$, in $B_1$, then  $x,x',y,v$ is a
$\{3,5,6,7,8\}$ or $\{3,5,6,7,9\}$-rainbow path between $x$ and $v$
under the above edge-coloring.

Otherwise, $x'$ has no neighbor in $B_1$. Then we know that
$N(x)\cap N(B_1)=\emptyset$ and $N(x')\cap N(B_1)\neq\emptyset$ by
the above coloring process. Pick $y\in B_1$ such that $N(x')\cap
N(y)\neq\emptyset$. Furthermore, pick $z\in N(x')\cap N(y)$. Then we
have $z\in A_1\cup B_1\cup Z$. Clearly, $x,x',z,y,v$ is a
$\{3,5,6,7,8\}$ or $\{3,5,6,7,9\}$-rainbow path between $x$ and $v$
under the above edge-coloring.
\end{proof}

To complete our edge-coloring, we further partition $J$ as follows:
\begin{align*}
&J_0=\{x\in J\ |\ x\ is\ not\ an\ isolated\ vertex\ in\ G[J]\},\\
&J_1=\{x\in J\setminus J_0\ |\ x\ has\ at\ least\ a\ neighbor\ in\ K\},\\
&J_2=\{x\in J\setminus (J_0\cup J_1)\ |\ x\ has\ at\ least\ a\ neighbor\ in\ I\},\\
&J_3=\{x\in J\setminus (J_0\cup J_1\cup J_2)\ |\ x\ has\ at\ least\ a\ neighbor\ in\ K\}\\
&J_4=J\setminus (J_0\cup J_1\cup J_2\cup J_3).
\end{align*}

Now we further color the edges of $G$ as follows: color the edges in
$E[Z,J_1\cup J_2\cup J_3]$ by color $7$; for any $x\in J_4$, color
one in $E[x,Z]$ by $7$, color the others in $E[x,Z]$ by $9$ (there
exists at least one such edge since $G$ is bridgeless).

To color the remaining edges, we need the following lemma.
\begin{lemma}
Let $S$ and $T$ be disjoint vertex sets of a graph $G$ such that
$S\subseteq N(T)$. If the induced subgraph $G[S]$ has no trivial
components, then there is an $\{\alpha,\beta,\gamma\}$-edge-coloring
of $G[S]\cup E[S, T]$, where the edges in $E[S, T]$ are colored by
colors $\alpha$ and $\beta$, and the edges in $G[S]$ are colored by
color $\gamma$, such that there exist two rainbow paths $P_1$ and
$P_2$ between $s$ and $T$ for every $s\in S$. Furthermore, if $P_1$
admits color $\{\alpha\}$, then $P_2$ admits colors
$\{\beta,\gamma\}$; if $P_1$ admits color $\{\beta\}$, then $P_2$
admits colors $\{\alpha,\gamma\}$.
\end{lemma}
\begin{proof}
Let $F$ be a spanning forest of $G[S]$, and let $X$ and $Y$ be any
one of the bipartition defined by this forest $F$. We give a
$3$-edge-coloring $c: E(G[S])\cup E[S,T]\rightarrow
\{\alpha,\beta,\gamma\}$ of $G$ by defining
$$c(e)=\left\{\begin{array}{ll} \alpha, & if\ e\in E[T,X];\\
\beta, & if\ e\in E[T,Y];\\ \gamma, & otherwise.
 \end{array}\right.$$
Clearly, for the above edge-coloring, there exist two rainbow paths
$P_1$ and $P_2$ between $s$ and $T$ for every $s\in S$. Furthermore,
if $P_1$ admits color $\{\alpha\}$, then $P_2$ admits colors
$\{\beta,\gamma\}$; if $P_2$ admits color $\{\beta\}$, then $P_2$
admits colors $\{\alpha,\gamma\}$.
\end{proof}

\begin{remark} In the next section, we shall use Lemma~1 repeatedly.
When we use this lemma, by ``forest'' we mean the forest defined in
Lemma~1 and by ``bipartition'' we mean the bipartition defined in
Lemma~1.
\end{remark}

We can complete our edge-coloring by Lemma~1. The edge-coloring in
Lemma~1 is called a $c_{\alpha,\beta}^{\gamma}$-edge-coloring of
$G[S]\cup E[S, T]$. Clearly, we can give $G[A_2]\cup E[A_2,u]$ and
$G[B_2]\cup E[B_2,v]$ a $c^8_{2,3}$-edge-coloring, $G[X_2]\cup
E[X_2,A_1]$ and $G[Y_2]\cup E[Y_2,B_1]$ a $c^7_{8,9}$-edge-coloring,
and $G[J_{4,2}]\cup E[J_{4,2},Z]$ a $c_{7,9}^8$-edge-coloring,
respectively. For the remaining edges, we do not color them at this
moment.

\subsection{Our upper bound}

In this subsection, we show that every bridgeless graph with
diameter~$3$ has a rainbow coloring with at most $9$ colors.

At first, we need a new notation. Let $X_1, X_2,\ldots\, X_k$ be
disjoint vertex subsets of $G$. The notation $X_1,X_2,\cdots,X_k$
means that there exists some desired rainbow path
$P=(x_1,x_2,\ldots,x_k)$, where $x_i\in X_i,\ i=1,2,\ldots,k$.
\begin{observation}
$(1)$ For every $x\in X_1$, there exists a $\{3,5,6,7\}$-rainbow
path between $x$ and $v$ under the above edge-coloring. $(2)$ For
every $x\in Y_1$, there exists a $\{2,4,6,7\}$-rainbow path between
$x$ and $u$ under the above edge-coloring.
\end{observation}
\begin{proof}
Since the proofs are similar, we only show $(1)$. For any $x\in
X_1$, by the definition of set $X_1$, we know that $x$ has neighbors
in $Y\cup Z\cup I\cup W$. Let $x'$ be a neighbor of $x$ in $Y\cup
Z\cup I\cup W$.

If $x'\in Y$, then $x'\in Y_1$ by the definition of set $Y_1$. Thus
$x,x',y,v$ is a $\{3,5,6,7\}$-rainbow path between $x$ and $v$ under
the above edge-coloring, where $y$ is a neighbor of $x'$ in $B_1$.

If $x'\in Z$, then $x,x',y,v$ is a $\{3,5,6,7\}$-rainbow path
between $x$ and $v$ under the above edge-coloring, where $y$ is a
neighbor of $x'$ in $B_1$.

If $x'\in I$, then $x,x',z,z',v$ is a $\{3,5,6,7\}$-rainbow path
between $x$ and $v$ under the above edge-coloring, where $z$ is a
neighbor of $x'$ in $Z$ and $z'$ is a neighbor of $z$ in $B_1$.

Otherwise, $x'\in W$, and then $x,x',y,y',v$ is a
$\{3,5,6,7\}$-rainbow path between $x$ and $v$ under the above
edge-coloring, where $y$ is a neighbor of $x'$ in $Y_1$ and $y'$ is
a neighbor of $y$ in $B_1$.
\end{proof}
\begin{observation}
$(1)$ For every $x\in A_1$, there exists a $\{3,4,5,6,7\}$-rainbow
path between $x$ and $v$ under the above edge-coloring. $(2)$ For
every $x\in B_1$, there exists a $\{2,4,5,6,7\}$-rainbow path
between $x$ and $u$ under the above edge-coloring.
\end{observation}
\begin{proof}
Since the proofs are similar, we only show $(1)$. For any $x\in
A_1$, by the definition of set $A_1$, we know that $x$ has neighbors
in $B_1 \cup Z \cup X_1$. Let $x'$ be a neighbor of $x$ in $Y\cup
Z\cup I\cup W$.

If $x'\in B_1$, then  $x,x',v$ is a $\{3,4,5,6,7\}$-rainbow path
between $x$ and $v$ under the above edge-coloring.

If $x'\in Z$, then $x,x',y,v$ is a $\{3,4,5,6,7\}$-rainbow path
between $x$ and $v$ under the above edge-coloring, where $y$ is a
neighbor of $x'$ in $B_1$.

Otherwise, $x'\in X_1$. By Observation~3, there exists a
$\{3,5,6,7\}$-rainbow path $P_{x',v}$. Thus $\{3,4,5,6,7\}$-rainbow
path between $x'$ and $v$ is obtained from $P_{x',v}$ and $xx'$.
\end{proof}

Before showing the following two lemmas, we need color some edges
that are not colored in the above coloring process: The edges of
$G[X_2]$ and $G[Y_2]$ that are not colored by color $7$ are colored
by color $4$.
\begin{lemma}
There exists a rainbow path between any two vertices of $X_2$ under
the above edge-coloring.
\end{lemma}
\begin{proof}
Let $x$ and $y$ be any two vertices in $X_2$. We consider the
following two cases.

{\itshape Case~1.} $x$ and $y$ are contained in different parts in
the bipartite of Lemma~1.

Let $x'$ and $y'$ be neighbors of $x$ and $y$ in $A_1$,
respectively, such that $xx'$ and $yy'$ are colored by $8$ and $9$,
respectively. By Observation~4, there exists a
$\{3,4,5,6,7\}$-rainbow path $P_{y',v}$. Thus, a rainbow path
between $x$ and $y$ is obtained from $P_{y',v}$ and $x,x',u,v$.

{\itshape Case~2.} $x$ and $y$ are contained in the same part. If
$d(x,B_1)=2$ or $d(y,B_1)=2$, without loss of generality, say
$d(x,B_1)=2$. Assume that $x'\in N(x)\cap N(B_1)$ and $xx'$ is
colored by color $8$. By the above set partition, we know $x'\in
A_1$. So, $x,x',z,v,A_1$ is $\{1,2,3,6,8\}$-rainbow, where $z$ is a
neighbor of $x'$ in $B_1$. By Lemma~1, $y$ and $A_1$ are connected
by a $\{7,9\}$-rainbow path. Thus, $x$ and $y$ are connected by a
rainbow path.

Suppose $d(x,B_1)=d(y,B_1)=3$. Let $x,x',z,w$ be a path between $x$
and $w$, where $w\in B_1$. By the above set partition, $x'\in
A_1\cup X_1\cup X_4$.

If $x'\in A_1$, then $z$ must be contained in $A_1\cup B_1\cup Z$.
Clearly $x',z,w$ is a $\{4,5,6\}$-rainbow path. Assume that $xx'$ is
colored by color $8$. Then $x,x',z,w,v,u,A_1$ is
$\{1,2,3,4,5,6,8\}$-rainbow. By Lemma~1, there exists a
$\{7,9\}$-rainbow path between $y$ and $A_1$. So, $x$ and $y$ are
connected by a rainbow path.

If $x'\in X_1$, then $z$ must be contained in $A_1\cup Z$. Clearly
$x',z,w$ is a $\{4,5,6\}$-rainbow path. Note that $xx'$ is colored
by color $7$. Thus $x,x',z,w,v,u,y',y$ is rainbow path between $x$
and $y$, where $y'$ is a neighbor of $y$ in $A_1$.

If $x'\in X_2$, then $z$ must be contained in $A_1$. If $xx'$ is
colored by color $4$. Assume that $x'z$ is colored by $8$. Then
$x,x',z,w,v,u,A_1$ is $\{1,2,3,4,6,8\}$-rainbow. By Lemma~1, there
exists a $\{7,9\}$-rainbow path between $y$ and $A_1$, without loss
of generality, denote it by $P_{y,y'}$. Thus, a rainbow path between
$x$ and $y$ is obtained from $x,x',z,w,v,u,y'$ and $P_{y',y}$. If
$xx'$ is colored by color $7$, then $xx'$ is contained in the
spanning forest in Lemma~1. Assume that $y'$ is a neighbor of $y$ in
$A_1$. Since $x'$ and $y$ are contained in different parts, $x'z$
and $yy'$ are colored by different colors in $\{8,9\}$. Thus,
$x,x',z,w,v,u,y',y$ is $\{1,2,3,6,7,8,9\}$-rainbow path between $x$
and $y$.
\end{proof}

Now, we color the edges in $E[Y_1,Y_2]$ as follows: For any $y\in
Y_2$, we color the edges in $E[y,Y_1]$ (if exists) by $9$ if the
edges in $E[y,B_1]$ are colored by $8$; we color the edges in
$E[y,Y_1]$ (if exists) by $8$ if the edges in $E[y,B_1]$ are colored
by $8$. Similar to Lemma~2, the following lemma holds.
\begin{lemma}
There exists a rainbow path between any two vertices of $Y_2$ under
the above edge-coloring.
\end{lemma}
\begin{lemma}
For any $x\in X_2$ and $y\in X_4$, there exists a rainbow path
between $x$ and $y$ under the above edge-coloring.
\end{lemma}
\begin{proof}
Let $x'$ be a neighbor of $x$ in $A_1$. If $xx'$ is colored by $8$.
By Observation~4, there exists a $\{3,4,5,6,7\}$-rainbow path
$P_{v,y'}$, where $y'$ is a neighbor of $y$ in $A_1$. Thus
$x,x',u,,v, P_{v,y'},y',y$ is a rainbow path between $x$ and $y$. If
$xx'$ is colored by $9$. By Observation~3, there exists a
$\{3,5,6,7\}$-rainbow path $P_{v,y'}$, where $y'$ is a neighbor of
$y$ in $X_1$. Thus $x,x',u,,v, P_{v,y'},y',y$ is a rainbow path
between $x$ and $y$.
\end{proof}

Similar to Lemma~4, the following lemma holds.
\begin{lemma}
For any $x\in Y_2$ and $y\in Y_4$, there exists a rainbow path
between $x$ and $y$ under the above edge-coloring.
\end{lemma}

Now, we are ready to prove Theorem~5.

\noindent{\itshape Proof of Theorem~5:} It suffices to show that the
edge-coloring of $G$ given in Subsection~$3.1$ is a rainbow
coloring, that is, for any pair of vertices $(x,y)\in (V(G),V(G))$
there exists a rainbow path between $x$ and $y$. Note that $A_1,B_1$
are nonempty, and the other sets may be empty. If some sets are
empty, then it is more easier. So we consider the most complicated
case, that is, all sets are nonempty. Consider the following three
cases.

{\itshape Case~1.} $(x,y)\in (\{u,v\},V(G))$.

It is easy to check that the conclusion holds from Figure~$1$ and
the definition of the vertex set partition.

{\itshape Case~2.} $(x,y)\in (A\cup B\cup X\cup Y\cup Z, V(G))$.

By Case~1, there exist rainbow paths between $A\cup B\cup X\cup
Y\cup Z$ and $\{y,u,\}$. Thus, we only consider $(x,y)\in (A\cup
X\cup B\cup Y\cup Z, V(G))$. Tables~$1$ and~$1'$ present a rainbow
path between $x$ and $y$ (We omit some rainbow paths, which are
denoted by ``$\cdots$", by symmetry).

In Table~$1$, $(x,y)\in (A_1,X_2)$ corresponds to ``$A_1,u,v,$~O2
and $A_1,X_2$''. By Figure~1 and the definition of the vertex set
partition, there exists a $\{1,2\}$-rainbow path between $x$ and
$v$. ``O2'' means that by Observation~2, there exists a
$\{3,4,5,6,7\}$-rainbow path between $v$ and $y'$, where $y'$ is a
neighbor of $y$ in $A_1$. Moreover, $y'y$ is colored by $8$ or $9$.
Thus $x$ and $y$ is connected by a rainbow path.

\begin{figure}[htbp]
{\tiny
\begin{center}
\renewcommand\arraystretch{1.5}
\begin{tabular}{|p{0.8cm}|p{1.27cm}|p{1.27cm}|p{1.27cm}|p{1.27cm}|p{1.27cm}|p{1.27cm}|p{1.27cm}|p{1.27cm}|}
\hline & $A_1$ & $A_2$ & $A_3$ & $X_1$ & $X_2$ & $X_3$ & $X_4$ & $Z$
\\\hline

$A_1$ & $A_1,u,v,$~and O4  & $A_1,u,$~and L1 & $A_1,u,A_3$ &
$A_1,u,v,$~and O3 & $A_1,u,v,$~O4 and $A_1,X_2$ & $A_1,u,v,$~O4
and~$A_1,X_3$ & $A_1,u,v,$~O4 and~$A_1,X_4$ & $A_1,u,v,B_1$, $Z$
\\\hline

$A_2$ & $\cdots$ & L1 & L1~and~$u,A_3$ & L1,~and~$u,A_1$, $X_1$ &
L1,~and~$u,A_1$, $X_2$ & L1,~and~$u,A_1$, $X_3$ & L1,~and~$u,A_1$,
$X_4$ & $A_2,u,v,B_1$, $Z$
\\\hline

$A_3$ & $\cdots$ & $\cdots$ & $A_3,A_1,u$, $A_3$ & $A_3,u,A_1$,
$X_1$& $A_3,u,A_1$, $X_2$ & $A_3,u,A_1$, $X_3$ & $A_3,u,A_1$, $X_4$
& $A_3,u,A_1$, $X_1,Z$
\\\hline

$X_1$ & $\cdots$ & $\cdots$ & $\cdots$ & $X_1,A_1,u$, $v$ and O3 &
$X_1,A_1,u$, $v$, O3 and L1 & $X_1,A_1,u$, $v$, O3 and $A_1,X_3$ &
$X_1,A_1,u$, $v$, O3 and $A_1,X_4$ & $X_1,A_1,u$, $v,B_1,Z$
\\\hline

$X_2$ & $\cdots$ & $\cdots$ & $\cdots$ & $\cdots$ &  L2 &
$X_2,A_1,u$, $v,B_1$,~O4 and~$A_1,X_3$ & L4 & $X_1,A_1,u$, $v,B_1,Z$
\\\hline

$X_3$ & $\cdots$ & $\cdots$ & $\cdots$ & $\cdots$ & $\cdots$ &
$X_3,A_1,u$, $v,B_1,$~O4 and~$A_1,X_3$ & $X_3,A_1,u$, $v,B_1,$~O4
and~$A_1,X_4$ & $X_3,A_1,u$, $v,B_1,Z$
\\\hline

$X_4$ & $\cdots$ & $\cdots$ & $\cdots$ & $\cdots$ & $\cdots$ &
$\cdots$ & $X_4,A_1,u$, $v,$~O3 and $X_1,X_4$ & $X_4,A_1,u$,
$v,B_1,Z$
\\\hline

$B_1$ & $B_1,v,u,A_1$ & $B_1,v,u,A_1$, $A_2$ & $B_1,v,u,A_1$, and L1
& $B_1,v,u,A_1$, $X_1$ & $B_1,v,u,A_1$, $X_2$ & $B_1,v,u,A_1$, $X_3$
& $B_1,v,u,A_1$, $X_4$ & $B_1,v,u,A_1$, $Z$
\\\hline

$B_2$ & L1,~and~$v,u$, $A_1$ & L1~and~$v,u,A_2$ & L1~and~$v$,
$u,A_3$ & L1,and~$v,u,A_1$, $X_1$ & L1,$v,u,A_1$, and L1 & L1,~and
$v,u$, $A_1,X_3$ & L1,~and $v,u$, $A_1,X_4$ & L1,~and $v,u$, $A_1,Z$
\\\hline

$B_3$ & $B_3,B_1,v,u$, $A_1$ & $B_3,v,u$, and L1 & $B_3,v,u,A_3$ &
$B_3,B_1,v,u$, $A_1,X_1$ & $B_3,B_1,v,u$, $A_1$~and~L1 &
$B_3,B_1,v,u$, $A_1,X_3$ & $B_3,B_1,v,u$, $A_1,X_4$ & $B_3,B_1,v,u$,
$A_1,$, $Z$
\\\hline

$Y_1$ & $Y_1,B_1,v,u$, $A_1$ & $Y_1,B_1,v,u$, and L1 &
$Y_1,B_1,v,u$, $A_1,A_3$ & $Y_1,B_1,v,u$, $A_1,X_1$ & $Y_1,B_1,v,u$,
$A_1,X_2$ & $Y_1,B_1,v,u$, $A_1,X_3$ & $Y_1,B_1,v,u$, $A_1,X_4$ &
$Y_1,B_1,v,u$, $A_1,Z$
\\\hline

$Y_2$ & $Y_2,B_1,v,u$, $A_1$ & $Y_2,B_1,v,u$, and L1 &
$Y_2,B_1,v,u$, $A_1,A_3$ & $Y_2,B_1,v,u$, $A_1,X_1$ & $Y_2,B_1,v,u$,
and L1 & $Y_2,B_1,v,u$, $A_1,X_3$ & L1,~and~$B_1,v$, $u,A_1,X_4$ &
$Y_2,B_1,v,u$, $A_1,Z$
\\\hline

$Y_3$ & $Y_3,B_1,v,u$, $A_1$ & $Y_3,B_1,v,u$, and L1 &
$Y_3,B_1,v,u$, $A_1,A_3$ & $Y_3,B_1,v,u$, $A_1,X_1$ & $Y_3,B_1,v,u$,
and L1 & $Y_3,B_1,v,u$, $A_1,X_3$ & $Y_3,B_1,v,u$, $A_1,X_4$ &
$Y_3,B_1,v,u$, $A_1,Z$
\\\hline

$Y_4$ & $Y_4,B_1,v,u$, $A_1$ & $Y_4,B_1,v,u$, and L1 &
$Y_4,B_1,v,u$, $A_1,A_3$ & $Y_4,B_1,v,u$, $A_1,X_1$ & $Y_4,B_1,v,u$,
and L1 & $Y_4,B_1,v,u$, $A_1,X_3$ & $Y_4,B_1,v,u$, $A_1,X_1,X_4$ &
$Y_4,B_1,v,u$, $A_1,Z$
\\\hline

$Z$ & $Z,B_1,v,u$, $A_1$ & $Z,B_1,v,u$, and L1 & $Z,B_1,v,u$,
$A_1,A_3$ & $Z,B_1,v,u$, $A_1,X_1$ & $Z,B_1,v,u$, $A_1,X_2$ &
$Z,B_1,v,u$, $A_1,X_3$ & $Z,B_1,v,u$, $A_1,X_4$ & $Z,B_1,v,u$,
$A_1,Z$
\\\hline

$I$ & $I,Z,B_1,v$, $u,A_1$ & $I,Z,B_1,v$, $u$ and L1 & $I,Z,B_1,v$,
$u,A_1,A_3$ & $I,Z,B_1,v$, $u,A_1,X_1$ & $I,Z,B_1,v$, $u,A_1,X_2$ &
$I,Z,B_1,v$, $u,A_1,X_3$ & $I,Z,B_1,v$, $u,A_1,X_4$ & $I,Z,B_1,v$,
$u,A_1,Z$
\\\hline

$J'=J\setminus J_0$ & $J',Z,B_1,v$, $u,A_1$ & $J',Z,B_1,v$, $u$, and
L1 & $J',Z,B_1,v$, $u,A_1,A_3$ & $J',Z,B_1,v$, $u,A_1,X_1$ &
$J',Z,B_1,v$, $u,A_1,X_2$ & $J',Z,B_1,v$, $u,A_1,X_3$ &
$J',Z,B_1,v$, $u,A_1,X_4$ & $J',Z,B_1,v$, $u,A_1,Z$
\\\hline

$J_0$ & $J_0,Z,B_1$, $v,u,A_1$ & $J_0,Z,B_1$, $v,u$, and L1&
$J_0,Z,B_1$, $v,u,A_1,A_3$ & $J_0,Z,B_1$, $v,u,A_1,X_1$ &
$J_0,Z,B_1$, $v,u,A_1,X_2$ & $J_0,Z,B_1$, $v,u,A_1,X_3$ &
$J_0,Z,B_1$, $v,u,A_1,X_1$, $X_4$ & $J_0,Z,B_1$, $v,u,A_1,Z$
\\\hline

$K$ & $K,Y_1,B_1$, $v,u,A_1$ & $K,Y_1,B_1$, $v,u$, and L1 &
$K,Y_1,B_1$, $v,u,A_1,A_3$ & $K,Y_1,B_1$, $v,u,A_1,X_1$ &
$K,Y_1,B_1$, $v,u,A_1,X_2$ & $K,Y_1,B_1$, $v,u,A_1,X_3$ &
$K,Y_1,B_1$, $v,u,A_1,X_4$ & $K,Y_1,B_1$, $v,u,A_1,Z$
\\\hline

$W$ & $W,Y_1,B_1$, $v,u,A_1$ & $W,Y_1,B_1$, $v,u$, and L1 &
$W,Y_1,B_1$, $v,u,A_1,A_3$ & $W,Y_1,B_1$, $v,u,A_1,X_1$ &
$W,Y_1,B_1$, $v,u,A_1,X_2$ & $W,Y_1,B_1$, $v,u,A_1,X_3$ &
$W,Y_1,B_1$, $v,u,A_1,X_4$ & $W,Y_1,B_1$, $v,u,A_1,Z$
\\\hline
\end{tabular}
\vspace*{40pt}

\centerline{\normalsize Table $1$. The rainbow paths in $G$ for Case
2.}
\end{center}}
\end{figure}

\begin{figure}[htbp]
{\tiny
\begin{center}
\renewcommand\arraystretch{1.4}
\begin{tabular}{|p{0.8cm}|p{1.50cm}|p{1.50cm}|p{1.50cm}|p{1.50cm}|p{1.50cm}|p{1.50cm}|p{1.50cm}|}
\hline & $B_1$ & $B_2$ & $B_3$ & $Y_1$ & $Y_2$ & $Y_3$ & $Y_4$
\\\hline

$B_1$ & $B_1,v,u,$~and O4  & $B_1,v,$~and L1 & $B_1,v,B_3$ &
$B_1,v,u,$~and O3 & $B_1,v,u,$~O4 and $B_1,Y_2$ & $B_1,v,u,$~O4
and~$B_1,Y_3$ & $B_1,v,u,$~O3 and~$Y_1,Y_4$
\\\hline

$B_2$ & $\cdots$ & L1 & L1~and~$v,B_3$ & L1,~and~$v,B_1,Y_1$ & L1,
$v,B_1$ and L1 & L1, $v,B_1,Y_3$ & L1, $v,B_1,Y_4$
\\\hline

$B_3$ & $\cdots$ & $\cdots$ & $B_3,B_1,v,B_3$ & $B_3,v,B_1,Y_1$&
$B_3,v,B_1,Y_2$ & $B_3,v,B_1,Y_3$ & $B_3,v,B_1,Y_4$

\\\hline

$Y_1$ & $\cdots$ & $\cdots$ & $\cdots$ &  O3~and~$u,v$, $B_1,Y_1$ &
O3,~and~$u,v$, $B_1,Y_2$ & O3,~and~$u,v$, $B_1,Y_3$ & O3,~and~$u,v$,
$B_1,Y_4$
\\\hline

$Y_2$ & $\cdots$ & $\cdots$ & $\cdots$ & $\cdots$ &  L3 &
$Y_2,B_1,v,u,$ O4~and~$B_1,Y_3$ & L5
\\\hline

$Y_3$ & $\cdots$ & $\cdots$ & $\cdots$ & $\cdots$ & $\cdots$ &
$Y_3,B_1,v,u,$ O4~and~$B_1,Y_3$ & $Y_4,Y_1$,~and, $v,B_1,Y_3$
\\\hline

$Y_4$ & $\cdots$ & $\cdots$ & $\cdots$ & $\cdots$ & $\cdots$ &
$\cdots$ & $Y_4,B_1,v,u,$ O3~and~$Y_1,Y_4$
\\\hline

$Z$ & $Z,A_1,u,v,B_1$ & $Z,A_1,u,v,$ and L1 & $Z,A_1,u,v,B_1$, $B_3$
& $Z,A_1,u,v,B_1$, $Y_1$ & $Z,A_1,u,v,B_1$, $Y_2$ & $Z,A_1,u,v,B_1$,
$Y_3$ & $Z,A_1,u,v,B_1$, $Y_4$
\\\hline

$I$ & $I,Z,A_1,u,v$, $B_1$ & $I,Z,A_1,u,v,$ and L1 & $I,Z,A_1,u,v$,
$B_1,B_3$ & $I,Z,A_1,u,v$, $B_1,Y_1$ & $I,Z,A_1,u,v$, $B_1,Y_2$ &
$I,Z,A_1,u,v$, $B_1,Y_3$ & $I,Z,A_1,u,v$, $B_1,Y_4$
\\\hline

$J'=J\setminus J_0$ & $J',Z,A_1,u,v$, $B_1$ & $J',Z,A_1,u,v,$ and L1
& $J',Z,A_1,u,v$, $B_1,B_3$ & $J',Z,A_1,u,v$, $B_1,Y_1$ &
$J',Z,A_1,u,v$, $B_1,Y_2$ & $J',Z,A_1,u,v$, $B_1,Y_3$ &
$J',Z,A_1,u,v$, $B_1,Y_4$
\\\hline

$J_0$ & $J_0,Z,A_1,u,v$, $B_1$ & $J_0,Z,A_1,u,v,$ and L1 &
$J_0,Z,A_1,u,v$, $B_1,B_3$ & $J_0,Z,A_1,u,v$, $B_1,Y_1$ &
$J_0,Z,A_1,u,v$, $B_1$ and L1 & $J_0,Z,A_1,u,v$, $B_1,Y_3$ &
$J_0,Z,A_1,u,v$, $B_1,Y_1,Y_4$
\\\hline

$K$ & $K,Z,A_1,u,v$, $B_1$ & $K,Z,A_1,u,v,$ and L1 & $K,Z,A_1,u,v$,
$B_1,B_3$ & $K,Z,A_1,u,v$, $B_1,Y_1$ & $K,Z,A_1,u,v$, $B_1,Y_2$ &
$K,Z,A_1,u,v$, $B_1,Y_3$ & $K,Z,A_1,u,v$, $B_1,Y_4$
\\\hline

$W$ & $W,Z,A_1,u,v$, $B_1$ & $W,Z,A_1,u,v,$ and L1 & $W,Z,A_1,u,v$,
$B_1,B_3$ & $W,Z,A_1,u,v$, $B_1,Y_1$ & $W,Z,A_1,u,v$, $B_1,Y_2$ &
$W,Z,A_1,u,v$, $B_1,Y_3$ & $W,Z,A_1,u,v$, $B_1,Y_4$
\\\hline
\end{tabular}

\vspace*{10pt} \centerline{\normalsize Table $1'$. The rainbow paths
in $G$ for Case 2.} \vspace*{10pt}
\begin{tabular}{|p{0.7cm}|p{1.27cm}|p{1.27cm}|p{1.27cm}|p{1.27cm}|p{1.27cm}|p{1.27cm}|p{1.27cm}|p{1.27cm}|}
\hline & $I$ & $K$ & $W$ & $J_1$ & $J_2$ & $J_3$ & $J_4$ & $J_0$
\\\hline

$I$ & $I,Z,B_1,v,u,$ $A_1,X_1,I$ & $I,Z,A_1,u,v$, $B_1,Y_1,K$ &
$I,X_1,A_1,u$, $v,B_1,Y_1,W$ & $I,X_1,A_1,u$, $v,B_1,Z,J_1$ &
$I,X_1,A_1,u$, $v,B_1,Z,J_2$ & $I,X_1,A_1,u$, $v,B_1,Z,J_3$ &
$I,X_1,A_1,u$, $v,B_1,Z,J_4$ & $I,X_1,A_1,u$, $v,B_1,Z,J_0$
\\\hline

$K$ & $\cdots$ & $K,Z,A_1,u$, $v,B_1,Y_1,K$ & $K,Z,A_1,u$,
$v,B_1,Y_1,W$ & $K,Z,A_1,u$, $v,B_1,Z,J_1$ & $K,Z,A_1,u$,
$v,B_1,Z,J_2$ & $K,Z,A_1,u$, $v,B_1,Z,J_3$ & $K,Z,A_1,u$,
$v,B_1,Z,J_4$ & $K,Z,A_1,u$, $v,B_1,Z,J_0$
\\\hline

$W$ & $\cdots$ & $\cdots$ & $W,X_1,A_1$, $u,v,B_1,Y_1$, $W$ &
$W,X_1,A_1$, $u,v,B_1,Z$, $J_1$ & $W,X_1,A_1$, $u,v,B_1,Z$, $J_2$ &
$W,X_1,A_1$, $u,v,B_1,Z$, $J_3$ & $W,X_1,A_1$, $u,v,B_1,Z$, $J_4$ &
$W,X_1,A_1$, $u,v,B_1,Z,J_0$
\\\hline

$J_1$ & $\cdots$ & $\cdots$ & $\cdots$ & $J_1,Z,A_1,u$,
$v,B_1,Z,J_1$ & $J_1,K,Z,B_1$, $v,u,A_1,Z$, $J_2$ & $J_1,K,Z,B_1$,
$v,u,A_1,Z$, $J_3$ & $J_1,K,Z,B_1$, $v,u,A_1,Z$, $J_4$ &
$J_1,K,Z,B_1$, $v,u,A_1,Z$, $J_0$
\\\hline

$J_2$ & $\cdots$ & $\cdots$ & $\cdots$ & $\cdots$ &  $J_2,I,X_1$,
$A_1,u,v,B_1$, $Z,J_2$ & $J_2,I,X_1$, $A_1,u,v,B_1$, $Z,J_3$ &
$J_2,I,X_1$, $A_1,u,v,B_1$, $Z,J_4$ & $J_2,I,X_1$, $A_1,u,v,B_1$,
$Z,J_0$
\\\hline

$J_3$ & $\cdots$ & $\cdots$ & $\cdots$ & $\cdots$ & $\cdots$ &
$J_3,W,X_1$, $A_1,u,v,B_1$, $Z,J_3$ & $J_3,W,X_1$, $A_1,u,v,B_1$,
$Z,J_4$ & $J_3,W,X_1$, $A_1,u,v,B_1$, $Z,J_0$
\\\hline

$J_4$ & $\cdots$ & $\cdots$ & $\cdots$ & $\cdots$ & $\cdots$ &
$\cdots$ & $J_0,Z,A_1$, $u,v,B_1,Z$, $J_0$ & $J_0,Z,A_1$,
$u,v,B_1,Z$, and L1
\\\hline

$J_0$ & $\cdots$ & $\cdots$ & $\cdots$ & $\cdots$ & $\cdots$ &
$\cdots$ & $\cdots$ & $J_0,Z,A_1$, $u,v,B_1,Z$, and L1
\\\hline
\end{tabular}
\vspace*{15pt}

\centerline{\normalsize Table $2$. The rainbow paths in $G$ for Case
3.}
\end{center}}
\end{figure}

{\itshape Case~3.} $(x,y)\in (I\cup J\cup K\cup W, V(G))$.

By Cases~1 and~2, there exists a rainbow path between every pair of
vertices in $(I\cup J\cup K\cup W, A\cup B\cup X\cup Y\cup
Z\cup\{u,v\})$. Thus, suppose $(x,y)\in (I\cup J\cup K\cup W, I\cup
J\cup K\cup W)$. Table~$2$ presents a rainbow path between $x$ and
$y$. \hfill$\sqcap\hskip-0.7em\sqcup$

We must say that at the moment we have not found examples showing
that the upper bound 9 is best possible. However, by a similar
method in \cite{DL} we can give the following example of graphs with
diameter 3 for which the rainbow connection number reaches 7.

\noindent{\bf Example~2.} Let $K_n$ be a complete graph with vertex
set $\{v_1,v_2,\ldots,v_n\}$, where $n\geq 217$. For every $v_i$, we
hang a path $P_i=(v_i,v_{i,1},v_{i,2},v_{i,3})$, and then we
identify the vertex $v_{i,3}$ with a vertex $v$. The resulting graph
is denoted by $G$. Clearly, $diam(G)=3$. Let $c$ be any
$6$-edge-coloring of $G$ with colors $\{1,2,\ldots,6\}$. Since
$6^3=216$ and there exist $217$ hanging paths $P_i$, at least two of
them have the same color. Without loss generality, say $P_1$ and
$P_2$, that is,
$c(v_1v_{1,1})=c(v_2v_{2,1}),c(v_{1,1}v_{1,2})=c(v_{2,1}v_{2,2})$
and $c(v_{1,2}v_{1,3})=c(v_{2,2}v_{2,3})$. By the structure of $G$,
it is easy to see that there exists no rainbow path between
$v_{1,1}$ and $v_{2,1}$ in $G$ under $c$. Thus $rc(G)\geq 7.$

\end{document}